\documentclass[a4paper,oneside]{amsart}
\def\uhr{\upharpoonright}
\usepackage{color}
\usepackage{amssymb,amsmath,amsthm,amsfonts}
\usepackage{stmaryrd}

\usepackage{comment} 
\excludecomment{note}

\theoremstyle{plain}
\newtheorem{theorem}{Theorem}[section]

\newtheorem{lemma}[theorem]{Lemma}
\newtheorem{corollary}[theorem]{Corollary}

\theoremstyle{definition}

\newtheorem{definition}[theorem]{Definition}
\newtheorem{remark}[theorem]{Remark}

\newcommand{\nat}{\mathbb{N}}
\newcommand{\aaa}{\alpha}

\newcommand{\TIME}{\mathsf{DTIME}}
\newcommand{\BPP}{\mathsf{BPP}}
\newcommand{\EXP}{\mathsf{EXP}}
\renewcommand{\P}{\mathsf{P}}
\newcommand{\PSPACE}{\mathsf{PSPACE}}

\renewcommand{\uhr}{\upharpoonright}

\newcommand{\la}{\langle}
\newcommand{\ra}{\rangle}
\newcommand{\sN}[1]{_{#1\in \nat}}
\begin{document}

\title{Closure
of resource-bounded randomness notions under polynomial time permutations}

\author[1]{Andr\'e Nies}
\address[1]{Department of Computer Science,
The University of Auckland, Private Bag 92019, Auckland, New Zealand,
\texttt{andre@cs.auckland.ac.nz}}

\author[2]{Frank~Stephan}
\address[2]{Department of Mathematics and Department of
Computer Science, National University of Singapore,
10 Lower Kent Ridge Road, Block S17, Singapore 119076,
Republic of Singapore, \texttt{fstephan@comp.nus.edu.sg}}

\date{\today}


\maketitle

\begin{abstract} An infinite bit sequence is called recursively random if no computable   strategy betting along the sequence has unbounded capital.  It is well-known that the property of recursive randomness   is closed
under computable permutations. We investigate analogous statements
for  randomness notions defined by betting strategies that are computable within resource bounds. Suppose that $S$ is a polynomial
time computable permutation of the set of strings over the unary
alphabet (identified with $\nat$). If the inverse of $S$  is not polynomially
bounded, it is easy to build a polynomial time random bit sequence
$Z$ such that $Z \circ S$ is not polynomial time random. So one
should only consider permutations $S$ satisfying the extra condition
that the inverse is polynomially bounded. Now the closure depends on
additional assumptions in complexity theory.

Our first result, Theorem~\ref{thm:BPP}, shows that if $\BPP$
contains a superpolynomial deterministic time class, such as
$\TIME(n^{\log n})$,
then polynomial time randomness is not preserved by some permutation 
$S$ such that in fact both $S$ and its inverse are in $\P$. Our
second result, Theorem~\ref{thm:scan}, shows that
polynomial space  randomness is preserved by polynomial time permutations
with polynomially boun\-ded inverse, so if $\P = \PSPACE$ then
polynomial time randomness is preserved. 
\end{abstract}


\section{Introduction} 
\noindent
Formal randomness notions for infinite bit
sequences can be studied via algorithmic tests. A hierarchy of such
notions has been introduced. See e.g.\ Downey
and Hirschfeldt
\cite{Downey.Hirschfeldt:book} 
or Nies \cite[Ch.\ 3]{Nies:book} for
definitions and basic properties, and also Li and Vit\'anyi
\cite{Li.Vitanyi:book}. Criteria for good randomness
notions include robustness under certain computable operations on the
bit sequences. In the simplest case, such an operation is a computable
permutation of the bits. For a permutation $S$ of $\nat$ and an
infinite bit sequence $Z$, identified with a subset of $\nat$, by $Z
\circ S$ we denote the sequence $Y$ such that $Y(n) = Z(S(n))$.
(Note that, when viewed as a subset of $\nat$, $Z \circ S$ equals $S^{-1}(Z)$.)
We say that a class $\mathcal C$ of bit sequences is \emph{closed 
under all members of a class $\mathcal G$ of permutations} if $Z
\in \mathcal C$ implies $Z \circ S \in \mathcal C$ for each $S \in
\mathcal G$.

A central notion of randomness was introduced by 
Martin-L\"of~\cite{Martin-Lof:66}. A Martin-L\"of test is a uniformly
$\Sigma^0_1$ sequence $\la G_m
\ra\sN m$ such that the uniform measure of $G_m$ is at most $2^{-m}$.
$Z$ fails such a test if $Z \in \bigcap_m G_m$; otherwise $Z$ passes
the test. $Z$ is Martin-L\"of random if it passes each such test. 
Clearly this randomness notion is closed under computable
permutations $S$: if $Z \circ S$ fails a Martin-L\"of-test $\la G_m
\ra\sN m$, then $Z$ fails the test $\la S^{-1}(G_m)\ra\sN m$. The
weaker notion of Schnorr randomness \cite{Schnorr:75}, where one also
requires that the measure of $G_m$ is a computable real uniformly in
$m$, is closed under computable permutations
by a similar argument. Recursive randomness \cite{Schnorr:75} (see
e.g.\ \cite[Ch.\ 7]{Nies:book} as a recent reference)
is defined via failure of all computable betting strategies (martingales),
rather than by a variant of Martin-L\"of's test notion.
Nonetheless, by a more involved argument, implicit in \cite[Section
4.1]{Buhrman.etal:00}, it is closed under
computable permutations. Also see Nies~\cite[Thm.\ 7.6.24]{Nies:book}
and Kjos-Hanssen, Nguyen and Rute~\cite{Kjos-Hanssen.Nguyen.Rute:14}.

Our main purpose is to study analogs in computational complexity
theory of this result. In order to guarantee compatibility with the
theory developed in Downey and Hirschfeldt~\cite{Downey.Hirschfeldt:book} and Nies~\cite{Nies:book} we view
  sets of numbers (i.e., infinite bit sequences), rather than sets of strings over an alphabet of size at least 2, as our principal
objects of study. 
We note that work of Lutz, Mayordomo, Ambos-Spies and others,
beginning in the 1980s and surveyed in Ambos-Spies and Mayordomo
\cite{Ambos.Mayordomo:97}, studied sets of strings: martingales bet on
the strings in length-lexicographical order. Such languages can be
identified with bit sequences via this order of strings, but the time
bounds imposed on martingales are   exponentially larger when
they bet on strings.

To be able to apply the notions of resource bounded computability to
bit sequences and permutations,
we will  identify infinite bit sequences with subsets of the set $\{0\}^*$
of unary strings. Such sets are called tally languages. We view
permutations as acting on $\{0\}^*$.
A bit sequence is \emph{polynomial
time random} if no polynomial time computable bettings strategy
succeeds on the sequence. This notion was briefly introduced by Schnorr~\cite{Schnorr:75},
studied implicitly in the above-mentioned work of Lutz, Mayordomo, Ambos-Spies and others, and in more explicit
form in Yongge Wang's 1996 thesis \cite{Wang:96}. 

Our leading question
is: \emph{under which polynomial time
computable permutations $S$ is polynomial time randomness closed?}
If $S^{-1}$ is not polynomially bounded, we build a 
polynomial time random bit sequence $Z$ such that $Z \circ S$ is not
polynomial time random. After that we assume that $S$ satisfies 
the extra condition that its inverse is polynomially bounded. Now
the closure depends on additional assumptions in complexity theory:
\begin{itemize}
\item The first result, Theorem~\ref{thm:BPP}, shows
that if $\BPP$
contains a superpolynomial deterministic time class, such as
$\TIME(n^{\log n})$,
then polynomial time randomness is not preserved by some
permutation $S$ such that both $S$ and its inverse are in $\P$. 

\item The second result, Theorem~\ref{thm:scan}, shows that
$\PSPACE$-ran\-dom\-ness is preserved by polynomial time permutations
with polynomially bounded inverse; so if $\P = \PSPACE$ then
polynomial time randomness is preserved by such permutations.
\end{itemize}
Broadly speaking, the idea for Theorem~\ref{thm:BPP}  is as follows. Choose an  $O(n^{\log n})$ time computable martingale $M$ only betting on odd positions  $1,3,5, \ldots$  that  dominates   (up to a positive factor) all polynomial time computable  martingales that only bets on odd  positions.   Use the hypothesis in order to take a language $A \in \BPP$    which tells at which extension  of a string of odd length $M$ does not increase.  Now let $B$ be a highly random set (albeit $B$ can be chosen in $E$). Let $Z$ be the bit sequence that copies $B(n)$ at position $2n$, and takes the value of $A$ at the string $Z(0)\ldots Z(2n)$ at position $2n+1$. Then one can verify that  $Z$ is polynomial time random. If  $\widehat Z$ is a  rearrangement of  the bits of $Z$ so that a sufficiently large block of  bits of $B$ is  interspersed between bits determined by $A$, then we can use these bits of $B$ as random bits required in a randomised polynomial time algorithm for $A$. This will show that $\widehat Z$ is not polynomial time random. 

Theorem~\ref{thm:scan} closely follows
Buhrman, van Melkebeek, Regan, Sivakumar and Strauss
\cite[Section 4.1]{Buhrman.etal:00}, which introduces and studies resource-bounded betting games. It actually shows that
$\PSPACE$-randomness is
closed under certain polynomial time scanning functions, which, unlike
permutations, can uncover the bits of a set in an order determined by
previous bits. Each permutation in question can be seen as a scanning
function of the appropriate kind. (We note that removing the  resource bounds from  Theorem~\ref{thm:scan} 
 yields a proof that recursive randomness is
closed under computable permutations, and in fact under computable scanning functions that scan each position.)
Thm.\ 5.6 in Buhrman et al.\ \cite{Buhrman.etal:00} is a related result based on the same methods developed there; however, in that result an assumption on the existence of certain pseudorandom generators is made, while our   Theorem~\ref{thm:scan} does not rest on any unproven assumptions. 

We note another notion of robustness for randomness notions. One can
easily adapt all the randomness notions to an alphabet
other than $\{0,1\}$. Base invariance says that the notion is
preserved when one replaces a sequence over one alphabet by a sequence
in a different alphabet that
denotes the same real number. Brattka, Miller and Nies
\cite{Brattka.Miller.ea:16} have shown this for recursive randomness,
and Figueira and Nies \cite{Figueira.Nies:13}
have shown it for polynomial time randomness, each time relying on the
connection of randomness of a real with differentiability at the real
of certain effective functions. 

Using Figueira and Nies \cite{Figueira.Nies:13}, 
Nies \cite{Nies:14} provides a characterisation of polynomial time
randomness for real numbers in terms of differentiability of all
polynomial time computable nondecreasing functions on the reals.

\section{Preliminaries} \label{s:prelim}
For a bound $h$, as usual $\TIME(h)$ denotes the languages $A$
computable in time $O(h)$. Informally we often say that $A$ is
computable in time $h$.  As in Ambos-Spies and Mayordomo~\cite{Ambos.Mayordomo:97}, we require that martingales have rational values.
\begin{definition} \label{df:mg}
A martingale $M$ is a function from $\{0,1\}^*$ to
$\{q \in {\mathbb Q}: q > 0\}$ satisfying $M(x) = (M(x0)+M(x1))/2$
for all $x \in \{0,1\}^*$. A martingale succeeds on a set $Z$
if $\limsup_n M(Z\uhr n)=\infty$. One says that a martingale
\emph{does not bet at
a position $n$} if $M(x0) = M(x1)$ for each $x \in \{0,1\}^n$.
\end{definition}

\noindent
One says that $Z$ is recursively random if no computable martingale
succeeds on $Z$.  

Each polynomial in this paper will be non-constant and have
natural number coefficients. For a polynomial time version of
recursive randomness, we have
to be careful how to define polynomial time computability for a
martingale: as in \cite{Buhrman.etal:00}, a positive rational number $q$ is presented by a pair $\la
k, n \ra$ of
denominator and numerator (written in binary) such that $q = k/n$ in
lowest terms. A martingale $M$ is
polynomial time computable if on input $x$ one can determine $M(x)$ in
this format in polynomial time. $Z$ is polynomial time random if no
such martingale succeeds on $Z$. In a similar way one defines
exponential time randomness.

A martingale $M$ is polynomial space computable if $M(x)$ can be
computed in polynomial space (including the space needed to write
the output). $Z$ is polynomial space random if no
such martingale succeeds on $Z$.

We first show that polynomial time randomness fails to be closed under 
polynomial time computable permutations $S$ that are ``dishonest'' in
the sense that $S(n)$ can be much less than $n$. 

\begin{theorem} \label{prop:S}
Let $S$ be a polynomial time
permutation of $\{0\}^*$ such that for each polynomial $p$, there are
infinitely many $n$ with $p(S(n)) \le n$. There is a polynomial time
random $Z$ computable in time $2^{O(n)}$ such that $Z \circ S$ is
not polynomial time random.
\end{theorem}

\noindent
Clearly a permutation $S$ as in Proposition~\ref{prop:S} exists: Let
$(p_k)$ list the non-constant polynomials with natural coefficients in
such a way that for $u \le n$, $O(n^2)$ steps suffice to verify
whether $p_k(u) \le n$.
On input $n$ of the form $\la k,i \ra$, see whether $p_k(\la k,0
\ra) \le n$. If not let $S(n) = \la k,i+1 \ra$. If so and $n$ is
least such, let $S(n) = \la k,0 \ra$. Otherwise $S(n) = n$. 

\begin{proof}[Proof of Proposition~\ref{prop:S}] Nies \cite[Section
7.4]{Nies:book} provided a construction template for recursively
random sets, going back to Schnorr's work. We adapt some
parts of this template to the resource bounded setting.

Let $\la B_k \ra$ be an effective listing of the polynomial time
martingales with positive rational values. We may assume that $B_k$
is computable in time $p_k(n)= k (n^k+1)$.

For each $n$, let $B_{k,n}$ be the martingale with initial capital
$1$ that does not bet until its input reaches length $n$, and then
uses the same betting factors as $B_k$. Thus, \[B_{k,n}(x) =
\frac{B_k(x)}{B_k(x \uhr n)}\] for any string $x$ of length at least $n$.
Let $\widetilde p_{k,n}$ be a polynomial so that $B_{k,n}(x)$ for $|x|
\ge n$ can be computed in time $\widetilde p_{k,n}(|x|)$.

We inductively define a sequence of numbers. Let $n_0 = 0$, and let 
$n_{k+1} $ be the least $n> n_k$ such that $q_k(S(n)+1) \le n$,
where $q_k$ is a polynomial time bound for the martingale $\sum_{r\le
k} 2^{-r} B_{r, n_{r}}$ and $q_k(n) \ge n+2$. 
Let $L= \sum_r 2^{-r} B_{r, n_{r}}$. Note that $L$ is a
rational-valued martingale, because on inputs of length at most
$n_{k}$, all the $B_{r, n_r}$ for $r > k$ together contribute $2^{-k}$.

Let now $Z$ be the left-most non-ascending path of $L$: $Z(m) = 0$ if
$L(Z \uhr m \hat {\ } 0) \le L(Z \uhr m)$, and $Z(m) = 1$ otherwise.
Since $L$ does not succeed on $Z$ and $L$ multiplicatively dominates
each $B_k$, the set $Z$ is polynomial time random.

Note that since $S \in \P$, from $n$ we can in polynomial time
recursively recover the sequence $n_0, q_0, n_1, q_1, \ldots$ and
thereby compute the maximal $k$ such that $n_k < n$. In particular we
can decide whether $n$ is of the form $n_{k+1}$ for some $k$. By
definition, for $n= n_{k+1}$ we have $q_k(S(n) +1) \le n$ and hence
$S(n) +1 < n_{k+1}$. 
Since $q_k $ as a time bound is sufficient to determine $L(y)$ for
strings $y$ of length $S(n)+1$, the bit $Z \circ S(n)$ can be
computed in time polynomial in $n$. Hence $Z \circ S$ is not
polynomially random.

We can ensure such a set $Z$ is computable in time $2^{O(n)}$ by choosing the
listing $\la B_k \ra$ appropriately. 
\end{proof}

\begin{remark} We note that methods involving the $\la B_{k,n} \ra$
similar to the above can be used to show that each class $\TIME(h)$
with superpolynomial time constructible $h$ contains a polynomial time
random (tally) set. We have to initiate a copy $\la B_{k,n} \ra$ of
$B_k$ finitely many times until a length $n$ is reached such that
for $m \ge n$, $h(m)$ time is sufficient to simulate its behaviour on
strings of length $m$. 
\end{remark}


\section{If $\BPP$ Contains a Superpolynomial Time Class Then Closure Fails}

\begin{definition} \label{def:FPC} A permutation $S$ of $\{0\}^*$ is
called  \emph{fully polynomial time computable} if both $S$ and $S^{-1}$ are
polynomial time computable.
\end{definition}

\noindent

A complexity theoretic assumption considerably  weaker than $\BPP =
\EXP$ suffices for non-closure.  


\begin{theorem} \label{thm:BPP}
Suppose that $\TIME(h) \subseteq \BPP$ for some time constructible
function $h$ that dominates all the polynomials. Then there are a 
polynomial time random set $Z \in \TIME(2^{3n})$
and a fully polynomial time computable permutation
$S$ such that $Z \circ S$ is not polynomial time random.
\end{theorem}

\begin{proof}
We may assume that $h(n) \le n^{\log n}$.
It is well-known that whenever a martingale
in a certain complexity class succeeds on a set $Z$ then there
is also a successful martingale in the same class betting only on even
positions, 
or there is a successful martingale betting only on the odd positions.

The construction has two steps. Firstly, by standard methods discussed
at the end of Section~\ref{s:prelim}, one can build a martingale $M$ in
$\TIME(h)$ which bets only on odd positions, and   dominates up to a
multiplicative constant all
polynomial time martingales betting on odd positions. Let
$$
   A = \{x \in \{0,1\}^*: x \text{ has odd length and } M(x1) < M(x0)\}.
$$
The set $A$ is in $\TIME(h)$ and hence by assumption in $\BPP$.

Secondly, let $B \subseteq \{0\}^*$ be a language on which no 
martingale in $\TIME(2^{4 \cdot n})$ succeeds.
Again by standard methods one can ensure that $B$ is in $\TIME(2^{5 \cdot n})$.
Define a set $Z \subseteq \mathbb N$ as follows:
$$
   Z(2n) = B(n); \ Z(2n+1) = A(Z \uhr 2n+1).
$$
\noindent We may visualise $Z$ as follows:

\medskip

 $  \begin{tabular}{|c|c|c|c|c|c|c|c|c| c|c|c|c| c}
\hline 

B &  {A} &   B  &   {A} & B  &  {A} & B&   {A}  &B  &   {A} & B &  {A}
 & B & \ldots \\

\hline

\end{tabular}$ 
\smallskip

  {\tiny $  B(0)  \,   { A(Z\uhr 1)}  \,  B(1 ) \,   {A (Z\uhr 3)}
\, B(2)  \,   {A (Z \uhr 5)}  \ldots$ }   

\medskip
\noindent
Clearly $Z \in \TIME(2^{3n})$. It is claimed that $Z$ is polynomial time random.
As the martingale $M$ only bets on odd positions, $Z$ is defined such
that $M$ never gains capital on $Z$. As $M$ is universal among the
martingales computable in polynomial time with this property, no
martingale betting on the odd positions succeeds on $Z$.

Suppose now that $L$ is a polynomial time martingale which bets on the even
positions and note that one can compute in time $O(h(n))$
from $B(0),B(1),\ldots,B(n)$ inductively the values $Z(0),Z(1),\ldots,Z(2n+1)$,
as for every $x$ of length $2n+1$ the value $A(x)$ can be computed in
time $h(n)$. Thus if $L$ succeeds on $Z$ then there is a new martingale $N$
succeeding on $B$ which satisfies that
$$
   N(B\uhr n+1) = L(Z\uhr 2n)
$$
and which uses that $Z(2n) = B(n)$ while the bits of $Z$ at odd positions
on which $L$ does not bet can be computed as indicated above from the
other bits. To compute $N(x)$ for $x$ of length $2n$ takes $q(n) +
\sum_{i< n} h(2i+1)$ steps for some polynomial $q$. So $N \in \TIME
(n^{O(\log n)})$, which contradicts the assumption that no such
martingale computable in time $O(2^{4n})$ succeeds on $B$. This
verifies the claim.

Since $A \in \BPP$, there is a polynomial $p$ such that
an appropriate randomised algorithm $\mathcal R$ on input $x \in
\{0,1\}^{2n+1}$ computes $A(x)$ in time $p(n)$,
with error probability $2^{-4n-2}$, using $p(n)$ random bits.
Now consider the sequence $\widehat Z$ consisting for $n=0,1,\ldots$
of $p(n)$ bits taken from $B$ followed by the bit $Z(2n+1)$. Again we
visualise $\widehat Z$:

\medskip

$ \begin{tabular}{|c|c|c|c|c|c|c|c|c| c|c|c|c| c}
\hline 

B &  {A} & B & B  & B &   {A}  &B  & B& B  &B  & B& B &  {A}  & \ldots \\

\hline

\end{tabular}$  

\smallskip

   $p(0)$ \hspace{1.1cm}  $p(1)$ \hspace {3cm} $p(2)$

\medskip
\noindent Formally one can define $\widehat Z$ from $Z$ as follows:
\begin{eqnarray*}
  \mbox{ for $m< p(n)$,} \hspace{1.5cm} \mbox{ } & & \\
  \widehat Z((\sum_{k<n} p(k))+n+m) & = &  B((\sum_{k<n} p(k))+m) \ = \ 
       Z(2(\sum_{k<n} p(k)+m)); \\
  \widehat Z((\sum_{k\leq n} p(k))+n) & = & Z(2n+1)= A( Z\uhr {2n+1}).
\end{eqnarray*}
This mapping is given by a permutation $S$ so that
$\widehat Z(r) = Z(S(r))$ for all positions $r$.
So if $r = (\sum_{k<n} p(k))+n+m$ then $S(r) = 2(\sum_{k<n} p(k))+2m$
and if $r = (\sum_{k\leq n} p(k))+n$ then $S(r) = 2n+1$, for all $m,n$
with $m < p(n)$. The permutation $S$ and its inverse satisfy that the
mappings $0^k \mapsto 0^{S(k)}$ and $0^k \mapsto 0^{S^{-1}(k)}$
on the unary strings $\{0\}^*$ are polynomial time computable, thus
the $S$ is of the form as required; to see this note that for a polynomial
$p$ also the mapping $n \mapsto \sum_{k<n} p(k)$ is a polynomial; similarly
for a function bounded by a polynomial.

Now it will be shown that $\widehat Z$ is not polynomial time random.
Note that there are $2^{2n+1}$ strings of length $2n+1$. Given a string
of $p(n)$ random bits, the probability that when using these bits the randomised
algorithm $\mathcal R$ computes $A(x)$ correctly
for all $x \in \{0,1\}^{2n+1}$ is at least $1-2^{2n+1} \cdot 2^{-4n-2}
= 1-2^{-2n-1}$.
We want to show that $B$ provides random bits that allow $\mathcal R$
to correctly compute $A$ for almost all inputs. Otherwise, we can
build a martingale $M$ computable in time $2^{10 \cdot n}$
which succeeds on $B$:
The martingale $M$ splits its
capital into bins of value $2^{-n-1}$ and for each block
of $p(n)$ bits starting at $\sum_{k < n} p(k)$, it takes the value
$2^{-n-1}$ from the corresponding bin and bets it on the strings $y$
consisting of 
$p(n)$ bits that do not compute all values of $A(x)$ with $x \in
\{0,1\}^{2n+1}$
correctly using $\mathcal R$. This condition can be checked for these bits
in the time bound given as it involves running $\mathcal R$ with $y$
as the random bits
on all strings $x$ of length $2n+1$ and comparing the result with $A(x)$ for all
$2^{p(n)}$ choices of random bits $y$. After these simulations, $M$
distributes the capital from the bin evenly on those
strings of random bits which cause $\mathcal R$ to make an error.
After having processed the
bits from the block of $p(n)$ bits, the capital in this bin remains
unchanged by future
bets. The set of random strings $y$ on which the computation of some of the
$A(x)$ in $x \in \{0,1\}^{2n+1}$ is false has at most
the probability $2^{-4n-2} \cdot 2^{2n+1} = 2^{-2n-1}$. Therefore
the capital from the bin multiplies at least by $2^{n+1}$ during the block and
reaches the value $1$.

For the time bound on $M$, whenever the input has length between
$\sum_{k<n} p(k)$ and $\sum_{k \leq n} p(k)$,
the martingale computes $2^{n+1}$ values $A(x)$ for
$x \in \{0,1\}^{2n+1}$ with respect to $p(n)$ random bits taking $2^{p(n)}$
possible choices. However, for all polynomials and almost all $n$,
$p(n) \leq \sum_{k < n} p(k)$, as the degree of the sum-polynomial of
$p$ is by one above the degree of $p$ and the polynomial $p$ is
positive. Thus, for such $n$, when $n' = \sum_{k<n} p(k)$ is a lower
bound on the
length of the input to the martingale $M$ then $p(n) \leq n'$ and
$2n+1 \leq n'$ and thus the whole computations can be handled in time
$O(2^{3n'})$.

If there are infinitely many blocks in $B$ where the random bits of this
block do not compute all $A(x)$ with $x$ of the corresponding length correctly,
then this martingale succeeds, contrary to the assumption on $B$.
So, for almost all $n$, the block of $p(n)$ random bits in $\widehat Z$ before
$A(Z \uhr {2n})$ permits to compute this value correctly.

Now this property will be used to
construct a polynomial time martingale $H$ which succeeds on $\widehat Z$. Let
$\tilde A(n)$ denote $A( Z\uhr {2n+1})$. Given $p(n)$ random bits
from $B$ preceding $\tilde A(n)$ in $\widehat Z$, the martingale $H$ archives
these bits without
betting on them. It then bets half of its capital on the value for
$\tilde A(n)$ computed from these random bits; note that due to
$\tilde A(0),\tilde A(1),\ldots,\tilde A(n-1)$ and $B(0),B(1),\ldots,B(n)$
being coded in $\widehat Z$ in positions before that of $\tilde A(n)$, when the
bet for $\tilde A(n) = Z(2n+1)$ has to be made, one can
retrieve besides the random bits also $Z(0)Z(1)\ldots Z(2n)$ from the
history. So one can use the random bits to compute the value almost
always correctly.
Thus the martingale $H$ will only finitely often place a wrong bet and lose
some of its capital, but for almost all $\tilde A(n)$ predict the value
correctly and multiply its capital by $3/2$. Thus the martingale
succeeds. As all the operations above are polynomial time computable, 
the set $\widehat Z$ is not polynomial time random.
\end{proof} 

\noindent
The proof of Theorem~\ref{thm:BPP} can be adjusted to obtain a corollary.

\begin{corollary} \label{cor:BPP}
Let $A,B \subseteq \{0\}^*$. Suppose that $A$ is in $\BPP$
and $B$ is $\EXP$-random relative to $A$.
Then $A$ is polynomial time computable relative to $B$,
and in particular not polynomial time random relative to $B$.
\end{corollary}

\begin{proof}
For the ease of notation, we often write $A(n)$ in place of $A(0^n)$ and so on;
however, both $A$ and $B$ are viewed as subsets of $\{0\}^*$.

There is a polynomial time algorithm and a polynomial $p$ such that
the algorithm uses $p(n)$ random bits to compute $A(n)$
with error probability $2^{-n}$. As in the theorem above,
one can now query $B$ for getting the random bits and the places
where the queries are asked are different for $n,m$ whenever $n \neq m$.
So there is a polynomial $q$ with $q(n)+p(n) = q(n+1)$ for all $n$
and where the algorithm asks the bits of $B$ at
$q(n),q(n)+1,\ldots,q(n)+p(n)-1$ to compute $A(n)$.

If now there is an error, then an exponential time martingale relative
to $A$ can make sufficient profit, as only a slim minority of the possiblities
of the bits of $B$ from $q(n)$ to $q(n)+p(n)-1$ are realised. This
contradicts the assumption that $B$ is random relative to $A$.
Hence $A$ can be computed relative to $B$ by this algorithm with
only finitely many errors; these can then be corrected by a finite
table holding the correct values for the positions where the algorithm
makes an error.
\end{proof}

\begin{remark}
In the proof of Theorem~\ref{thm:BPP}, $Z = \tilde A \oplus B$
is polynomial time random; however, $\tilde A$ is not polynomial time
random relative to $B$, as the rearrangement with $S$ shows.
Note that van Lambalgen's Theorem
\cite{Lambalgen:90} says that in a recursion-theoretic setting, 
$\tilde A \oplus B$ is random iff (a) $B$ is random and
(b) $\tilde A$ is random relative to $B$.
Thus, under the assumption that $\BPP = \EXP$, one of the directions of
the van Lambalgen Theorem does
not hold for polynomial time randomness.

The corollary also shows that one can choose, under the assumption
that $\BPP $ contains a superpolynomial time class, sets $A,B
\subseteq \{0\}^*$ such that $A$ is
polynomial time random,
$B$ is polynomial time random relative to $A$ and $A$ is polynomial
time computable relative
to $B$. Hence this assumption implies that $A$ is a basis for 
polynomial time randomness
even though $A$ is polynomial time random itself. This contrasts with the
setting of Martin-L\"of randomness randomness in recursion theory: a basis for 
Martin-L\"of randomness has to
be trivial and therefore cannot be random
\cite{Hirschfeldt.Nies.Stephan:07,Nies:05}. On the other hand, the
bases for recursive randomness include every
set below the halting problem that is not diagonally noncomputable (DNC),
but no set of PA degree \cite{Hirschfeldt.Nies.Stephan:07}. Every high set computes a recursively random set, and an incomplete  high r.e.\ set is not DNC. So a recursively random set can   be a basis for recursive randomness. 
\end{remark}

\section{If $\P = \PSPACE$ Then Closure Holds}

\noindent
We say that $Z \subseteq \nat $ is \emph{polynomial space random} if no
  martingale computable in polynomial space  succeeds on $Z$. In this section we
show that polynomial space randomness is closed under fully polynomial time
computable permutations in the sense of Definition \ref{def:FPC}. If
$\P = \PSPACE$ this closure property applies to polynomial time
randomness as well.

In fact we show a stronger closure property where the permutations
are generalised to certain non-monotonic scanning rules (adaptively
specifying an order in which bits are read). We modify the argument given
by Buhrman, van Melkebeek, Regan, Sivakumar and Strauss
\cite[Section 4.1]{Buhrman.etal:00}, which was not concerned with
polynomial space randomness, but rather was geared to the context of Lutz's
theory of resource bounded measure. As already mentioned, in that
theory, the positions a
martingale bets on
are strings in some non-unary alphabet. 
Such strings can be suitably
encoded by natural numbers; however, the resource bounds change
when one converts such a martingale into one in the sense of our
Definition~\ref{df:mg}. 

\begin{definition} A \emph{scanning function} is a function $V \colon
\{0,1\}^* \to \{0\}^*$ such that $V(\alpha) \neq V(\alpha \uhr i)$ for
each $\alpha \in \{0,1\}^* $ and each $i < |\alpha|$. In the context of $V$,
we will call a string $\alpha$ a \emph{run of $V$}, thinking of
$\aaa$ as a sequence of answers to oracle queries. We will call
$V(\alpha \uhr i)$ the \emph{$i$-th query in the run of
$V$ on $\alpha$}. 

As before, subsets of $\nat$ will be identified with languages over
the unary alphabet~$\{0 \}$. For $Z \subseteq \nat$ let $Z \circ V
\subseteq \nat$ be the set $Y$ such that $Y(i) = Z(V (Y \uhr i))$
for each~$i$. In the following we review some key technical concepts
\cite[Section 4.1]{Buhrman.etal:00}, somewhat changing the
terminology in order to make it compatible with the one of Nies \cite[Section
7.5]{Nies:book} where non-monotonic randomness notions are studied.

For a function $g \colon \nat \to \nat$, one says that $V$ is
\emph{$g$-filling} if for each $n$ and each run $\alpha$ of length $g(n)$, 
\[ \forall r < n \, \exists i \, V(\alpha \uhr i) = r. \]
\end{definition}

\begin{definition}
A \emph{non-monotonic betting strategy} $G$ is
a pair $(V,B)$ such that $V$ is a scanning function and $B$ is a
martingale. $G$ succeeds on $Z\subseteq \{0\}^*$ if $\lim_n
B(Z \circ V\uhr n) = \infty$. 

One says that a non-monotonic betting strategy $G$ is computable in
polynomial space if both $V$ and $B$ are computable in polynomial
space. One says that $Z\subseteq \nat $ is \emph{non-monotonically
polynomial space random} if no such betting strategy succeeds on $Z$.
\end{definition}

\noindent
The final concept we need is that of consistency between a run
$\alpha$ of $V$ and a string~$w$.

\begin{definition}
For bit strings $\alpha, w$, we write $\alpha \sim_V w $ if for
each $j < |\alpha|$, if the $j$-th query $x$ in the run of $V$ on
$\alpha$ is less than $|w|$, then $w(x) = \alpha(j)$. 


\begin{lemma} Suppose $V$ is $g$-filling. Let $|\alpha| \ge i:=
g(|w|)$. Then $\alpha \sim_V w $ iff $\alpha \uhr i \sim_V w$.
\end{lemma}

\noindent
To see this, note that any query $q$ with $q< |w|$ has to be
asked before stage $g(|w|)$ by the definition of the function $g$.
\end{definition}

\begin{theorem} \label{thm:scan}
Let $V$ be a scanning function in $\PSPACE$ that is $g$-filling for
a polynomial bound $g$. If $Z$ is polynomial space random, then so is
$Z \circ V$.
\end{theorem}

\begin{proof}
Suppose $Z \circ V$ is not polynomial space random. Let $G=(V,B)$
be a betting strategy in $\PSPACE$ that succeeds on $Z$;
thus, $B$ succeeds on $Z \circ V$.

We define a martingale $D$ in $\PSPACE$ that succeeds on $Z$.
We may assume that $g(n) \ge n$. For $t \ge g(|w|)$ let 
\[ D(w) = 2^{|w|-t} \sum_{|\aaa| = t \ \land \ \aaa \sim_V w} B(\aaa).\]
By the claim above and since $B$ is a martingale, this definition
is independent of $t$. Note that among the runs $\alpha$ of length
$t$, a fraction of $2^{-|w|}$ satisfy that $\alpha \sim_V w$; so
$D(w)$ is simply the average value of $B(\aaa)$ over all such $\aaa$. 

If we let $t = g(|w|)$, by the hypotheses that $G$ is in $\PSPACE$
and that $g$ is a polynomial, $D$ is in $\PSPACE$.

The rest of the argument somewhat simplifies the one of
\cite{Buhrman.etal:00} in the present context.

\begin{lemma}
$D$ is a martingale.
\end{lemma} 

\noindent Let $w$ be a string of length $n$. If $|\aaa| = g(n+1)$ and
$\alpha \sim_V w$, then either $\aaa \sim_V w0$ or $\aaa \sim_V w1$.
Letting $u= g(n+1)$, for each $r =0,1$ we have
\begin{center}
  $D(wr) = 2^{|w|+1 -u} \sum_{|\aaa| = u \ \land \ \aaa \sim_V {wr} }
           B(\aaa)$.
\end{center} 
Hence, since the definition of $D(w)$ does not depend on the choice
of $t\ge g(|w|)$, 
\begin{center}
$D(w0) + D(w1) =  2^{|w|+1 -u} \sum_{|\aaa| = u \
\land \ \aaa \sim_V {w} } B(\aaa) = 2 D(w)$.
\end{center} 

\begin{lemma} $D$ succeeds on $Z$. \end{lemma} 

\noindent
We may assume that $B(x)>0$ for each $x$. The Savings Lemma (see
e.g.\ Nies \cite[7.1.14]{Nies:book})
states that each computable martingale $M$ can be turned into a
computable martingale $\widehat M$ that succeeds on the same sets, and
has the extra property that $\widehat M(\beta) \ge \widehat M(\alpha)
-2$ for each strings $\beta \supseteq \alpha$ (namely, $\widehat M$
never loses more than $2$). It is easy to see from the proof that if
$M$ is computable in polynomial space, then so is $\widehat M$. So we
may assume that $B$ has this property.

This implies that for each 
$c\in \nat$ there is a prefix $\alpha$ of $Z \circ V$ such that 
\begin{center}
  $B(\beta ) \ge c$ for each string $\beta \succeq \aaa$.
\end{center}
By definition of $Z \circ V$ we have $\aaa(i) = Z(V(\aaa\uhr i))$ for each
$i < |\aaa|$. Let $r = 1+ \max_{i< |\aaa|} V(\aaa\uhr i)$ be $1+$
the maximum query asked in the run of $V$ on $\aaa$, and let $w = Z
\uhr r$. So $g(r) \ge |\aaa|$. 

If $\beta\sim_V w$ is a string such that $|\beta| =g(r)$, then $\beta
\succeq \aaa$, for $\aaa(r) \neq \beta(r)$ for some $r < |\aaa|$
would imply that $\beta \not \sim_V w$ as $w$ answers all such 
queries correctly. So $B(\beta) \ge c$. Hence $D(w) \ge c$ because
$D(w)$ is the average over values $B(\beta)$ for all such $\beta$.
\end{proof}

\begin{corollary} Let $S$ be a polynomial time computable permutation
of $\{0\}^*$ such that $S^{-1}$ is polynomially bounded. If $Z$ is
polynomial space random, then so is $Z \circ S$. 
\end{corollary}

\begin{proof} 
The permutation $S$ can be viewed as a scanning function $V_S$ that only
looks at the length of the input: $V_S(\aaa) = S(|\aaa|)$.
By hypothesis on $S$, the scanning function $V_S$
is polynomially filling. So $Z\circ S = Z \circ V_S$ is polynomial
space random by the theorem.
\end{proof} 

\noindent
The foregoing corollary can be restated in terms of randomness on
languages in the sense of \cite{Ambos.Mayordomo:97}: Let $S$ be a
exponential time computable permutation
of $\{0,1\}^*$ such that $|S^{-1}(x)| = O(|x|)$ for each string $x$.
If a language $Z$ is exponential space random, then so is $Z \circ S$. 

We end with a question. Recall that $\mathsf{PP}$ denotes
probabilistic polynomial time, a subclass of $\PSPACE$.
If $\P= \mathsf{PP}$, is polynomial time randomness closed
under permutations $S$ of $\{0\}^*$ such that $S, S^{-1}$ are polynomial time
computable?

\medskip
\noindent
{\bf Acknowledgments.} The authors would like to thank Eric Allender, Klaus Ambos-Spies, Alexander Galicki, Elvira Majordomo, 
and Wolfgang Merkle for discussions and comments.

{A.~Nies is supported
in part by the Marsden Fund of the Royal Society of New Zealand, UoA
13-184. F.~Stephan
is supported in part by the Singapore Ministry of Education Academic
Research Fund Tier 2
grant MOE2016-T2-1-019 / R146-000-234-112. Part of this work was done
while F.~Stephan
was on sabbatical leave at the University of Auckland. The work was
completed while  Nies  visited the Institute for Mathematical Sciences
at NUS during the  2017 programme ``Aspects of Computation''.}

\def\cprime{$'$} \def\cprime{$'$}

\end{document}